\documentclass[11pt,twoside]{amsart}
\usepackage{amsxtra}
\usepackage{color}
\usepackage{amsopn}
\usepackage{amsmath,amsthm,amssymb,arydshln}
\usepackage{mathrsfs,mathtools}
\usepackage{stmaryrd}
\usepackage{hyperref}
\usepackage{enumerate}
\usepackage{xcolor}
\usepackage{pifont}
\usepackage{adjustbox}

\newtheorem{theorem}{Theorem}[section]

\newtheorem{proposition}[theorem]{Proposition}
\newtheorem{lemma}[theorem]{Lemma}
\newtheorem*{MT}{Main Theorem}
\theoremstyle{definition}
\newtheorem{definition}[theorem]{Definition}

\newtheorem{remark}[theorem]{Remark}

\newcommand{\R}{{\mathbb R}}

\newcommand{\C}{{\mathbb C}}

\newcommand{\W}{\wedge}

\newcommand{\f}{\varphi}

\newcommand{\GL}{{\rm GL}}

\newcommand{\G}{{\rm G}}

\newcommand{\SL}{{\rm SL}}
\newcommand{\Gd}{{\rm G}_2}

\newcommand{\frg}{\mathfrak{g}}

\newcommand{\frh}{\mathfrak{h}}

\newcommand{\frn}{\mathfrak{n}}
\newcommand{\frd}{\mathfrak{d}}
\newcommand{\fre}{\mathfrak{e}}
\newcommand{\frp}{\mathfrak{p}}

\newcommand{\frs}{\mathfrak{s}}
\newcommand{\frz}{\mathfrak{z}}
\newcommand{\frr}{\mathfrak{r}}
\newcommand{\fru}{\mathfrak{u}}

\newcommand{\frsl}{\mathfrak{sl}}
\newcommand{\frso}{\mathfrak{so}}
\newcommand{\frsu}{\mathfrak{su}}
\newcommand{\aff}{\mathfrak{aff}}

\newcommand{\Der}{{\rm Der}}
\newcommand{\sst}{\scriptscriptstyle}

\numberwithin{equation}{section}

\title[Closed G$_2$-structures on non-solvable Lie groups]{Closed G$_{\mathbf2}$-structures on non-solvable Lie groups}
\author{Anna Fino and Alberto Raffero}
\thanks{The authors were supported by GNSAGA of INdAM}
\address{Dipartimento di Matematica ``G. Peano'' \\ Universit\`a degli Studi di Torino\\
Via Carlo Alberto 10\\
10123 Torino\\ Italy}
\email{annamaria.fino@unito.it, alberto.raffero@unito.it}

\begin{document}
\maketitle
\begin{abstract}
We investigate the existence of left-invariant closed G$_2$-structures on seven-dimensional non-solvable Lie groups, providing the first examples of this type. 
When the Lie algebra has trivial Levi decomposition, we show that such a structure exists only when the semisimple part is isomorphic to $\frsl(2,\R)$ and 
the radical is unimodular and centerless. Moreover, we classify unimodular Lie algebras with non-trivial Levi decomposition admitting closed G$_2$-structures. 
\end{abstract}


\section{Introduction}
A $\G_2$-structure on a seven-dimensional smooth manifold $M$ is a $\G_2$-reduction of the structure group of its frame bundle. 
It is well-known that such a reduction exists if and only if $M$ is orientable and spin, and that it  
is characterized by the existence of a 3-form $\f$ on $M$ satisfying a certain nondegeneracy condition. 
This 3-form gives rise to a Riemannian metric $g_\f$ with volume form $dV_\f$ via the identity  
\[
g_\f(X,Y)\,dV_\f = \frac16\,\iota_X\f\W\iota_Y\f\W\f,
\]
for any pair of vector fields $X,Y$ on $M.$ 

A $\G_2$-structure $\f$ is said to be {\em closed} if the defining 3-form satisfies the equation $d\f=0$, 
while it is called {\em co-closed} if $d*_\f\f=0$, $*_\f$ being the Hodge operator defined by $g_\f$ and $dV_\f$.  
When both of these conditions hold, the intrinsic torsion of the $\G_2$-structure vanishes identically, the Riemannian metric $g_\f$ is Ricci-flat, 
and $\mathrm{Hol}(g_\f)\subseteq\G_2$ (cf.~\cite{Bry,FeGr}). 

Differently from the co-closed case, where the existence of co-closed $\G_2$-structures on every seven-dimensional spin manifold was proved in \cite{CrNo}, 
general results on the existence of closed $\G_2$-structures on compact 7-manifolds are still not known.

The first example of compact manifold endowed with a closed $\G_2$-structure was constructed by Fern\'andez \cite{Fer}, 
and it consists of the compact quotient of a simply connected nilpotent Lie group by a lattice, namely a co-compact discrete subgroup. 
The classification of nilpotent Lie algebras admitting closed $\G_2$-structures was achieved by Conti and Fern\'andez \cite{CoFe}. 
As the simply connected nilpotent Lie group corresponding to each one of them admits a lattice, their result provides additional locally homogeneous compact examples. 

In the solvable non-nilpotent case, the first compact example was described in \cite{Fer1}, while a classification result for almost  
abelian Lie algebras was obtained by Freibert \cite{Fre}. 
Further solvable examples satisfying a distinguished curvature property are discussed in \cite{Bry,Lau2}. 
  
Complete closed $\G_2$-structures which are invariant under the cohomogeneity one action of a compact simple Lie group are exhibited in \cite{ClSw}. 
Recently, Podest\`a and the second named author proved that there are no compact homogeneous spaces endowed with an invariant closed (non-flat) $\G_2$-structure \cite{PoRa}. 

Remarkably, up to now nothing is known about the existence of closed $\G_2$-structures on seven-dimensional non-solvable Lie algebras.  
The aim of the present work is to investigate whether examples of this type occur.   

As almost Hermitian geometry and $\G_2$-geometry correspond  to the geometry of 1-fold and 2-fold vector cross products, respectively (cf.~\cite{Gra}), 
closed $\G_2$-structures may be regarded as the seven-dimensional analogue of almost K\"ahler structures on even-dimensional manifolds. 
Consequently, it is natural to ask whether known restrictions on the existence of symplectic forms on Lie algebras are valid for closed $\G_2$-structures, too. 

In contrast to the results on symplectic Lie algebras \cite{Chu,LicMed}, we show that there exist non-solvable unimodular Lie algebras admitting closed $\G_2$-structures.  
The unimodular case is of foremost interest when one is looking for compact examples, since being unimodular is a necessary condition for the existence of lattices \cite{Mil}.

Before stating our main result, we fix some notations used throughout the paper. 
Given a Lie algebra $\frg$ of dimension $n$, we write its structure equations with respect to a basis of $1$-forms $\{e^1,\ldots,e^n\}$  by 
specifying the $n$-tuple $(de^1,\ldots,de^n)$, $d$ being the Chevalley-Eilenberg differential of $\frg$. 
Moreover, we use the shorthand $e^{ijk\cdots}$ to denote the wedge product  $e^i\W e^j\W e^k\W\cdots$.

\begin{MT}
Up to isomorphism, the only seven-dimensional, non-solvable unimodular Lie algebras admitting a closed $\G_2$-structure are the following:
\begin{eqnarray*}
&\left(-e^{23},-2e^{12},2e^{13},0,-e^{45},\frac{1}{2}e^{46}-e^{47},\frac{1}{2}e^{47}\right);&  \\ 
&\left(-e^{23},-2e^{12},2e^{13},0,-e^{45},-\mu\,e^{46},(1+\mu)\,e^{47}\right), ~ -1<\mu\leq-\frac12;& \\ 
&\left(-e^{23},-2e^{12},2e^{13},0, -\mu\,e^{45},\frac{\mu}{2}\,e^{46}-e^{47},e^{46}+\frac{\mu}{2}\,e^{47}\right),~\mu>0;& \\ 
&\left(-e^{23},-2e^{12},2e^{13},-e^{14}-e^{25}-e^{47},e^{15}-e^{34}-e^{57},2e^{67},0\right).& 
\end{eqnarray*}
In the above list, all of the Lie algebras but the last one have a trivial Levi decomposition of the form $\frg=\frs\oplus\frr$, 
where the semisimple part is isomorphic to $\frsl(2,\R)$, and the radical $\frr$ is centerless. 
\end{MT}

The proof of the main theorem is made up by various results, which are stated and proved separately in order to make the presentation more clear and to emphasize the peculiarities 
of each case under consideration. 
In section $\S$\ref{NonSolv7}, we summarize all possible seven-dimensional non-solvable Lie algebras up to isomorphism, 
distinguishing between those with non-trivial and trivial Levi decomposition. 
Then, we deal with the various possibilities in sections $\S$\ref{STrLevi} and $\S$\ref{SNoTrLevi}. 
The proof of the theorem follows combining lemmas \ref{LemmaSemiS6}, \ref{LemmaCenterNontrivial},  
and propositions \ref{nonexistenceso3}, \ref{nonexistencesl2}, \ref{existencesl2}, \ref{UNTN} and \ref{UNTY}. 

In addition to this result, in section $\S$\ref{STrLevi} we also show that non-unimodular Lie algebras with trivial Levi decomposition cannot admit closed $\G_2$-structures. 

Having obtained the first examples of non-solvable Lie algebras admitting closed $\G_2$-structures, 
it would be interesting to study the behaviour of the Laplacian $\G_2$-flow \cite{Bry,LotWei} in this setting, in order to describe similarities and differences with the solvable case 
\cite{FeFiMa,FiRa,Lau1,Lau2}.  
We plan to do this in a subsequent work.

The computations in the proofs of propositions \ref{nonexistenceso3}, \ref{nonexistencesl2}, \ref{existencesl2}, \ref{UNTN}, \ref{UNTY} have been done with the aid of the software Maple 18 
and its package {\em difforms}.

\section{Closed $\G_2$-structures on Lie algebras}
Let $\mathfrak{g}$ be a seven-dimensional real Lie algebra.
Every 3-form $\phi\in\Lambda^3(\mathfrak{g}^*)$ on $\mathfrak{g}$ gives rise to a symmetric bilinear map 
\[
b_\phi:\mathfrak{g}\times\mathfrak{g}\rightarrow\Lambda^7(\mathfrak{g}^*),\quad (v,w) \mapsto \frac16\, \iota_v\phi\W\iota_w\phi\W\phi.
\] 
By \cite{Hit}, the $\GL(\frg)$-orbit of $\phi$ is open in $\Lambda^3(\mathfrak{g}^*)$ if and only if $\det(b_\phi)^\frac{1}{9}\in\Lambda^7(\frg^*)$ is nonzero. 
When this happens, $\phi$ is said to be {\em stable}. Using this notion, it is possible to give the following definition (cf.~\cite{CLSS,Hit}). 

\begin{definition}
A $\G_2$-structure on $\frg$ is a stable 3-form $\f\in\Lambda^3(\mathfrak{g}^*)$ for which the bilinear map $g_\f:\frg\times \frg\rightarrow\R$ defined by
\[
g_\f\coloneqq \det(b_\f)^{-\frac19}b_\f,
\] 
is positive definite. The volume form induced by $\f$ is $dV_\f=\det(b_\f)^\frac{1}{9}$. 
\end{definition}

A $\G_2$-structure $\f$ on $\frg$ is called {\em closed} if $d\f=0$, where $d$ denotes the Chevalley-Eilenberg differential of $\frg$. 
Clearly, left multiplication allows to extend any closed $\G_2$-structure $\f$ on a Lie algebra $\frg$ to a left-invariant closed $\G_2$-structure on every corresponding Lie group. 
Conversely, any left-invariant closed $\G_2$-structure $\f$ on a Lie group $\G$ is determined by the closed $\G_2$-structure $\f_{\sst{1_\G}}$ 
on $T_{\sst{1_\G}}\G\cong\frg$.

In the next lemma, we summarize some obstructions to the existence of closed $\G_2$-structures on a Lie algebra. 
The proof is an immediate consequence of the properties of the defining 3-form.
\begin{lemma}\label{obstrgen}
A seven-dimensional oriented real Lie algebra $\mathfrak{g}$ does not admit any closed $\Gd$-structure if for every closed 3-form $\phi\in\Lambda^3(\mathfrak{g}^*)$ 
one of the following conditions hold for the map $b_\phi:\mathfrak{g}\times\mathfrak{g}\rightarrow\Lambda^7(\mathfrak{g}^*)\cong\R$:
\begin{enumerate}[i)]
\item there exists $v\in\mathfrak{g}\smallsetminus\{0\}$ such that $b_\phi(v,v)=0$;
\item there exist  $v,w\in\mathfrak{g}\smallsetminus\{0\}$ such that $b_\phi(v,v)\,b_\phi(w,w)\leq0$.
\end{enumerate} 
This result does not depend on the choice of the orientation. 
\end{lemma}

Further obstructions can be obtained exploiting known non-existence results for symplectic forms on even-dimensional Lie algebras (e.g.~\cite{Chu,LicMed}) and the following. 
\begin{proposition}[\cite{CoFe}]\label{ClosedSympl}
Assume that $\frg$ has non-trivial center and that there exists a closed $\G_2$-structure $\f$ on it. 
If $\pi:\frg\rightarrow\frh$ is a Lie algebra epimorphism with kernel contained in the center, and $\dim(\frh)=6$, then $\frh$ admits a symplectic form given by $\iota_\xi\f$, 
where $\langle\xi\rangle=\ker(\pi)$. 
\end{proposition}

\section{Non-solvable seven-dimensional Lie algebras}\label{NonSolv7}
Consider a seven-dimensional real Lie algebra $\frg$ and denote by $\frr$ its radical, i.e., its maximal solvable ideal. 
When $\frg$ is neither semisimple nor solvable, it is well-known that there exists a semisimple subalgebra $\frs\subseteq\frg$ such that $\frg$ is a semidirect product 
\begin{equation}\label{LeviDec}
\frg=\frs\oplus_\rho\frr,
\end{equation}
for a suitable homomorphism $\rho:\frs\rightarrow\Der(\frr)$ (see e.g.~\cite[Thm.~B.2]{Kna}). 
The decomposition \eqref{LeviDec} is called {\em Levi decomposition} of $\frg$. 
If $\rho$ is the zero map, $\frg$ reduces to the product algebra $\frg=\frs\oplus\frr$, and the Levi decomposition is said to be {\em trivial}. 

By \cite{Tur}, there are seven non-isomorphic irreducible Lie algebras of dimension seven with non-trivial Levi decomposition. 
We recall their structure equations with respect to a basis $\left\{e^1,\ldots,e^7\right\}$ of their dual algebra in Table \ref{tabLnt} (cf.~\cite[Table II]{Tur}). 
\begin{table}[ht]
\centering
\renewcommand\arraystretch{1.4}
\adjustbox{max width=\textwidth}{
\begin{tabular}{|c|c|}
\hline
$\frg$				& 	$\left(de^1,de^2,de^3,de^4,de^5,de^6,de^7\right)$					 	\\ \hline \hline
$L_{7,1}$				&	$\left(-e^{23},e^{13},-e^{12},-e^{26}+e^{35}-e^{47},e^{16}-e^{34}-e^{57},-e^{15}+e^{24}-e^{67},0\right)$					\\ \hline    

$L_{7,2}$				&	$\left(-e^{23},e^{13},-e^{12},\frac12 e^{17}+\frac12 e^{25}+\frac12 e^{36},\frac12 e^{16}-\frac12 e^{24}-\frac12 e^{37}, 
								-\frac12 e^{15}+\frac12 e^{27}-\frac12 e^{34},-\frac12 e^{14}-\frac12 e^{26}+\frac12 e^{35}\right)$				\\ \hline

$L^a_{7,3}$			&	$\left(-e^{23},-2e^{12},2e^{13},-e^{14}-e^{25}-e^{47},e^{15}-e^{34}-e^{57},-a\,e^{67},0\right),\quad a\neq0$				\\ \hline
						
$L_{7,4}$				&	$\left(-e^{23},-2e^{12},2e^{13},-e^{14}-e^{25}-e^{47},e^{15}-e^{34}-e^{57},-e^{45}-2e^{67},0\right)$						\\ \hline

$L_{7,5}$				&	$\left(-e^{23},-2e^{12},2e^{13},-2e^{14}-2e^{25}-e^{47},-e^{26}-e^{34}-e^{57},2e^{16}-2e^{35}-e^{67},0\right)$				\\ \hline

$L_{7,6}$				&	$\left(-e^{23},-2e^{12},2e^{13},-3e^{14}-3e^{25},-e^{15}-2e^{26}-e^{34},e^{16}-e^{27}-2e^{35},3e^{17}-3e^{36}\right)$		\\ \hline

$L_{7,7}$				&	$\left(-e^{23},-2e^{12},2e^{13},-e^{14}-e^{25},e^{15}-e^{34},-e^{16}-e^{27},e^{17}-e^{36}\right)$							\\ \hline
\end{tabular}}
\vspace{0.1cm}
\caption{Seven-dimensional irreducible Lie algebras with non-trivial Levi decomposition.}\label{tabLnt}
\end{table}
\renewcommand\arraystretch{1}

\begin{remark}\label{RemUnimNonTrLevi}
It is straightforward to check that the only unimodular Lie algebras appearing in Table \ref{tabLnt} are $L_{7,2}$, $L^{-2}_{7,3}$, $L_{7,6}$, and $L_{7,7}$.
\end{remark}

In addition to the Lie algebras summarized in Table \ref{tabLnt}, there exist four non-isomorphic unimodular reducible Lie algebras of dimension seven with non-trivial Levi decomposition. 
They are given by the product of the unimodular Lie algebras appearing in \cite[Table I]{Tur}, namely $L_{5,1}$, $L_{6,1}$, $L_{6,2}$, $L_{6,4}$,  
and the abelian Lie algebra of suitable dimension. 
We collect their structure equations in Table \ref{tabLnt2}.

\begin{table}[ht]
\centering
\renewcommand\arraystretch{1.4}
\adjustbox{max width=\textwidth}{
\begin{tabular}{|c|c|}
\hline
$\frg$				& 	$\left(de^1,de^2,de^3,de^4,de^5,de^6,de^7\right)$					 	\\ \hline \hline
$L_{5,1}\oplus\R^2$		&	$\left(-e^{23},-2e^{12},2e^{13}, -e^{14}-e^{25}, e^{15}-e^{34}, 0, 0\right)$					\\ \hline    

$L_{6,1}\oplus\R$		&	$\left(-e^{23},e^{13},-e^{12}, -e^{26}+e^{35}, e^{16}-e^{34}, -e^{15}+e^{24}, 0\right)$				\\ \hline
				
$L_{6,2}\oplus\R$		&	$\left(-e^{23},-2e^{12},2e^{13}, -e^{14}-e^{25}, e^{15}-e^{34}, -e^{45}, 0\right)$						\\ \hline

$L_{6,4}\oplus\R$		&	$\left(-e^{23},-2e^{12},2e^{13}, -2e^{14}-2e^{25}, -e^{26}-e^{34}, 2 e^{16}- 2 e^{35}, 0\right)$		\\ \hline
\end{tabular}}
\vspace{0.1cm}
\caption{Seven-dimensional unimodular reducible Lie algebras with non-trivial Levi decomposition.}\label{tabLnt2}
\end{table}
\renewcommand\arraystretch{1}

If the Levi decomposition $\frg=\frs\oplus\frr$ is trivial, then either $\dim(\frs)=3$ or $\dim(\frs)=6$.  
Indeed, it is known that a semisimple Lie algebra has dimension at least three, and that there are no semisimple Lie algebras of dimension four, five and seven. 

Recall that there exist only two non-isomorphic three-dimensional semisimple Lie algebras, 
namely $\frsl(2,\R)$ and $\frso(3)\cong\frsu(2)$. With respect to a suitable basis $\left\{e^1,e^2,e^3\right\}$ of their dual algebras, their structure equations are 
\begin{eqnarray}
\frsl(2,\R)	&=&	\left(-e^{23},-2e^{12},2e^{13}\right),\label{sl2}\\
\frso(3)	&=&	\left(-e^{23},e^{13},-e^{12}\right).\label{so3}
\end{eqnarray}

Up to isomorphism, a semisimple Lie algebra $\frs$ of dimension six is either $\frso(3,1)$ or a product $\frs=\frs'\oplus\frs''$, where  
the summands $\frs'$ and $\frs''$ are three-dimensional and semisimple.

\section{Non-solvable Lie algebras with trivial Levi decomposition}\label{STrLevi}
In this section, we investigate the existence of closed $\G_2$-structures on seven-dimensional non-solvable Lie algebras with trivial Levi decomposition. 
As observed in section $\S$\ref{NonSolv7}, we have to consider Lie algebras of the form $\frg=\frs\oplus\frr$, with $\dim(\frs)=3$ or $\dim(\frs)=6$. 

We claim that the case $\dim(\frs)=6$ can be excluded using Proposition \ref{ClosedSympl} together with a result by Chu \cite[Thm.~8]{Chu}.  
We recall it in the next theorem, giving an alternative shorter proof. 
\begin{theorem}[\cite{Chu}]\label{ChuSympSemi}
Any semisimple Lie group has no left-invariant symplectic structure. 
\end{theorem}
\begin{proof}
As we are considering left-invariant symplectic structures on Lie groups, it is sufficient to prove the assertion for semisimple Lie algebras.  
Let $\frp$ be a semisimple Lie algebra of even dimension, and assume that there exists a symplectic form $\omega\in\Lambda^2(\frp^*)$ on it. 
Denote by $\mathcal{B}$ the Cartan-Killing form of $\frp$. Since it is non-degenerate, there exists a $\mathcal{B}$-skew-symmetric endomorphism $F\in\mathrm{End}(\frp)$ such that 
$\omega(v,w)=\mathcal{B}(Fv,w)$ for all $v,w\in\frp$. The closedness of $\omega$ implies that $F$ is a derivation of $\frp$, i.e., $F[v,w] = [Fv,w]+[v,Fw]$ for all $v,w\in\frp$. 
As $\frp$ is semisimple, there exists some nonzero $z\in\frp$ such that $F=\mathrm{ad}_z$. 
Consequently, $\omega(z,v)=0$ for all $v\in\frp$. This is clearly in contrast with the nondegeneracy of $\omega$. 
\end{proof}

We now show our claim. 
\begin{lemma}\label{LemmaSemiS6}
If the Levi decomposition $\frg=\frs\oplus\frr$ is trivial, and $\dim(\frs)=6$, then $\frg$ does not admit any closed $\G_2$-structure.  
\end{lemma}
\begin{proof}
We have $\frg=\frs\oplus\R$. 
Assume that $\frg$ admits a closed $\G_2$-structure $\f$, and let $\xi$ be a generator of the radical $\frr=\R$. 
Then, the 2-form $\iota_\xi\f$ is a symplectic form on $\frs$ by Proposition \ref{ClosedSympl}. This contradicts Theorem \ref{ChuSympSemi}. 
\end{proof}

Hence, we are left with the case $\frg=\frs\oplus\frr$, with $\dim(\frs)=3$ and $\dim(\frr)=4$. 
To deal with it, we begin recalling the classification of four-dimensional solvable Lie algebras obtained in \cite{Mub} (see also \cite{AnBaDoOv}). 
\begin{theorem}[\cite{AnBaDoOv,Mub}]\label{solvable4}
Let $\frr$ be a four-dimensional solvable Lie algebra. Then, $\frr$ is isomorphic to one and only one of the following:
\begin{eqnarray*}
\R^4					&=&	\left(0,0,0,0\right),\\
\aff(\R)\oplus\aff(\R)		&=&	\left(0,-e^{12},0,-e^{34}\right),\\
\R\oplus\frh_3			&=& \left(0,0,0,-e^{23}\right),\\
\R\oplus\frr_3			&=& \left(0,0,-e^{23}-e^{24},-e^{24}\right),\\
\R\oplus\frr_{3,\lambda}	&=& \left(0,0,-e^{23},-\lambda e^{24}\right),\quad |\lambda|\leq1,\\
\R\oplus\frr'_{3,\lambda}	&=& \left(0,0,-\lambda e^{23}-e^{24},e^{23}-\lambda e^{24}\right),\quad \lambda\geq0,\\
\frn_4				&=& \left(0,0,-e^{12},-e^{13}\right),\\
\aff(\C)				&=& \left(0,0,-e^{13}+e^{24},-e^{14}-e^{23}\right),\\
\frr_4					&=& \left(0,-e^{12}-e^{13},-e^{13}-e^{14},-e^{14}\right),\\
\frr_{4,\lambda}			&=& \left(0,-e^{12},-\lambda e^{13}-e^{14},-\lambda e^{14}\right),\\
\frr_{4,\mu,\lambda}		&=& \left(0,-e^{12},-\mu e^{13},-\lambda e^{14}\right),~\lambda\mu\neq0,~-1<\mu\leq\lambda\leq1~or~-1=\mu\leq\lambda<0,\\
\frr'_{4,\mu,\lambda}		&=& \left(0,-\mu e^{12},-\lambda e^{13}-e^{14},e^{13}-\lambda e^{14}\right),\quad \mu>0,\\
\frd_4				&=& \left(0,-e^{12},e^{13},-e^{23}\right),\\
\frd_{4,\lambda}		&=& \left(0,-\lambda e^{12},(\lambda-1)e^{13},-e^{14}-e^{23}\right),\quad \lambda\geq\frac12,\\
\frd'_{4,\lambda}		&=& \left(0,-\lambda e^{12}-e^{13},e^{12}-\lambda e^{13},-2\lambda e^{14}-e^{23}\right),\quad \lambda\geq0,\\
\frh_4				&=& \left(0,-e^{12}-e^{13},-e^{13},-2e^{14}-e^{23}\right).
\end{eqnarray*}
In the above list, the unimodular Lie algebras are the following: $\R^4$, $\R\oplus\frh_3$, $\R\oplus\frr_{3,-1}$, $\R\oplus\frr'_{3,0}$, $\frn_4$, $\frr_{4,-\frac12}$, 
$\frr_{4,\mu,-1-\mu}$ with $-1<\mu\leq-\frac12$, $\frr'_{4,\mu,-\frac{\mu}{2}}$ with $\mu>0$, $\frd_4$, $\frd'_{4,0}$. 
\end{theorem}

Using the next theorem, we can exclude the existence of closed $\G_2$-structures on Lie algebras with trivial Levi decomposition and whose radical has non-trivial center. 
\begin{theorem}[\cite{Chu}]\label{ChuSympLeviTriv}
Suppose that the Lie algebra of a symplectic Lie group $\mathrm{H}$ has trivial Levi decomposition. Then, $\mathrm{H}$ is solvable. 
\end{theorem} 
More in detail, if $\f$ is a closed $\G_2$-structure on $\frg=\frs\oplus\frr$ and $\xi\in\frz(\frr)$ is a nonzero vector in the center of $\frr$, 
then $\frr'\coloneqq\frr\slash\langle\xi\rangle$ is still solvable and the 
contraction $\iota_\xi\f$ defines a symplectic form on $\frs\oplus\frr'$ by Proposition \ref{ClosedSympl}. This is clearly not possible unless $\frs=\{0\}$ by Theorem \ref{ChuSympLeviTriv}. 
Consequently, we have the following. 
\begin{lemma}\label{LemmaCenterNontrivial}
Let $\frg=\frs\oplus\frr$ with $\dim(\frs)=3$, and assume that the center $\frz(\frr)$ of $\frr$ is non-trivial. Then, $\frg$ does not admit any closed $\G_2$-structure. 
\end{lemma}

A computation on a case-by-case basis allows to establish which four-dimensional solvable Lie algebras have non-trivial center. 
\begin{lemma}\label{centerrad}
The four-dimensional solvable Lie algebras with non-trivial center are $\R^4$, $\R\oplus\frh_3$, $\R\oplus\frr_3$, $\R\oplus\frr_{3,\lambda}$, $\R\oplus\frr'_{3,\lambda}$,
$\frn_4$, $\frr_{4,0}$, $\frd_4$, $\frd'_{4,0}$.
\end{lemma}

Combining the above lemma with the next result, we get that the case $\frg=\frso(3)\oplus\frr$ never occurs in the main theorem.   
\begin{proposition}\label{nonexistenceso3}
Let $\frs=\frso(3)$ and let $\frr$ be a four-dimensional solvable Lie algebra with $\frz(\frr)=\{0\}$. 
Then, the Lie algebra $\frg=\frso(3)\oplus\frr$ does not admit any closed $\G_2$-structure. 
\end{proposition}
\begin{proof}
Consider the Lie algebra $\frg=\frso(3)\oplus\frr$, where $\frr$ is a four-dimensional solvable Lie algebra with trivial center. 
Let $\left\{e_1,e_2,e_3\right\}$ be the basis of $\frso(3)$ whose dual basis $\left\{e^1,e^2,e^3\right\}$ defines the structure equations given in \eqref{so3}.  
Similarly, denote by $\left\{e^4,e^5,e^6,e^7\right\}$ the basis of $\frr^*$ for which $\frr$ has the structure equations appearing in Theorem \ref{solvable4}, and let 
$\left\{e_4,e_5,e_6,e_7\right\}$ be the corresponding basis of $\frr$. 
Then, $\left\{e^1,\ldots,e^7\right\}$ is a basis of $\frg^*$. 
The expression of the generic closed 3-form on $\frg$ can be obtained starting with a generic 3-form
\[
\phi = \sum_{1\leq i<j<k\leq7} \phi_{ijk}e^{ijk}\in\Lambda^3(\frg^*),
\]
and solving the linear system in the variables $\phi_{ijk}$ arising from the equation $d\phi=0$ (see Appendix \ref{appendix} for more details). 
Once we have computed such a 3-form, we consider the associated bilinear map $b_\phi:\frg\times\frg\rightarrow\Lambda^7(\frg^*)\cong\R$,  
where the isomorphism $\Lambda^7(\frg^*)\cong\R$ is obtained by fixing the volume form $e^{1234567}$. 
Evaluating $b_\phi$ on the basis vectors, we immediately see that the obstructions given in Lemma \ref{obstrgen} hold in the following cases:
\begin{eqnarray*}
\aff(\R)\oplus\aff(\R):						& &	b_\phi(e_5,e_5)b_\phi(e_7,e_7)\leq0,\\
\aff(\C):								& &	b_\phi(e_6,e_6)b_\phi(e_7,e_7)\leq0,\\
\frr_4:								& &	b_\phi(e_i,e_i)=0,~i=5,6,7,\\
\frr_{4,\lambda},~\lambda\neq0:			& & 	b_\phi(e_i,e_i)=0,~i=5,6,7, \mbox{ when } \lambda\neq-\frac12,\\
									& & 	b_\phi(e_5,e_5)b_\phi(e_6,e_6)\leq0, \mbox{ when } \lambda=-\frac12,\\
\frr_{4,\mu,\lambda}:						& & 	b_\phi(e_i,e_i)=0,~i=5,6,7, \mbox{ when } \lambda+\mu+1\neq0,\\
									& &	b_\phi(e_5,e_5)b_\phi(e_7,e_7)\leq0, \mbox{ when } \lambda=-\mu-1,~-1<\mu\leq-\frac12, \\
\frr'_{4,\mu,\lambda}:						& & 	b_\phi(e_i,e_i)=0,~i=5,6,7,\mbox{ when } \mu+2\lambda\neq0,\\
\frd_{4,\lambda}:						& &	b_\phi(e_7,e_7)=0, \\
\frd'_{4,\lambda},~\lambda>0:				& & 	b_\phi(e_7,e_7)=0,\\				
\frh_4:								& &	b_\phi(e_7,e_7)=0.
\end{eqnarray*}
We are left with the case $\frr=\frr'_{4,\mu,-\frac{\mu}{2}}$ with $\mu>0$. Here, since
\[
b_\phi(e_5,e_5) = -\mu(\phi_{125}^2+\phi_{135}^2+\phi_{235}^2)\phi_{567},
\]
we must have $\phi_{567}<0$. Now, adding $b_\phi(e_6,e_6)$ and $b_\phi(e_7,e_7)$ gives
\[
\frac18 \mu\, \phi_{567} \left( (\mu\, \phi_{127}-2\phi_{347})^2+(\mu\,\phi_{137}+2\phi_{247})^2+4\phi_{127}^2+4\phi_{137}^2
+4\phi_{236}^2+4\phi_{237}^2 \right).
\]
This implies that $b_\phi(e_6,e_6)+b_\phi(e_7,e_7)\leq0$. Hence,  $b_\phi(e_6,e_6)b_\phi(e_7,e_7)\leq 0$. 
\end{proof}

We now consider Lie algebras with trivial Levi decomposition and semisimple part isomorphic to $\frsl(2,\R)$.

\begin{proposition}\label{nonexistencesl2}
The Lie algebra $\frg=\frsl(2,\R)\oplus\frr$ does not have any closed $\G_2$-structure when $\frr$ is centerless and non-unimodular. 
\end{proposition}
\begin{proof}
We have to consider Lie algebras of the form $\frg=\frsl(2,\R)\oplus\frr$, where the radical $\frr$ is one of the following: 
$\aff(\R)\oplus\aff(\R)$, $\aff(\C)$, $\frr_4$,  $\frr_{4,\lambda}$ with $\lambda\neq0,-\frac12$, $\frr_{4,\mu,\lambda}$ with $\lambda+\mu+1\neq0$, 
$\frr'_{4,\mu,\lambda}$ with $\mu+2\lambda\neq0$, $\frd_{4,\lambda}$, $\frd'_{4,\lambda}$ with $\lambda>0$, $\frh_4$. 
Fix a basis $\left\{e^1,\ldots,e^7\right\}$ of $\frg^*$ in a similar way as in the proof of Proposition \ref{nonexistenceso3}, 
where $\left\{e^1,e^2,e^3\right\}$ is the basis of $\frsl(2,\R)^*$ for which the structure equations are those given in \eqref{sl2}. 
Let $\phi\in\Lambda^3(\frg^*)$ be a generic closed 3-form on $\frg$, and fix the volume form $e^{1234567}$. 
Then, evaluating the bilinear form $b_\phi$ on the basis vectors of $\frg$, we obtain the following obstructions:
\begin{eqnarray*}
\aff(\C):								& &	b_\phi(e_6,e_6)b_\phi(e_7,e_7)\leq0,\\
\frr_4:								& &	b_\phi(e_i,e_i)=0,~i=5,6,7,\\
\frr_{4,\lambda},~\lambda\neq0:			& & 	b_\phi(e_i,e_i)=0,~i=5,6,7,\mbox{ when } \lambda\neq -\frac12,\\
\frr_{4,\mu,\lambda}:						& & 	b_\phi(e_i,e_i)=0,~i=5,6,7,\mbox{ when } \lambda+\mu+1\neq0,\\
\frr'_{4,\mu,\lambda}:						& & 	b_\phi(e_i,e_i)=0,~i=5,6,7,\mbox{ when } \mu+2\lambda\neq0,\\
\frd_{4,\lambda}:						& & 	b_\phi(e_7,e_7)=0,\\
\frd'_{4,\lambda},~\lambda>0:				& & 	b_\phi(e_7,e_7)=0,\\
\frh_4:								& &	b_\phi(e_7,e_7)=0.
\end{eqnarray*}
We still have to examine the case $\frr=\aff(\R)\oplus\aff(\R)$. Here, we focus on the restriction of $b_\phi$ to the subspace of $\frg$ spanned by $e_6$ and $e_7$. 
The $2\times2$ symmetric matrix associated with this bilinear form has non-positive determinant. 
Hence, denoted by $\lambda_1,\lambda_2$ its real eigenvalues, either at least one of them is zero or $\lambda_1\lambda_2<0$. 
In the first case, the eigenvector $v$ corresponding to the eigenvalue 0 satisfies $b_\phi(v,v)=0$. 
Otherwise, if $w_1$ is an eigenvector with positive eigenvalue and $w_2$ is an eigenvector with negative eigenvalue, we have $b_\phi(w_1,w_1)>0$ and 
 $b_\phi(w_2,w_2)<0$. 
\end{proof}

\begin{remark}
The obstructions to the existence of closed $\G_2$-structures in the proofs of Proposition \ref{nonexistenceso3} and Proposition \ref{nonexistencesl2} 
always involve some basis vectors of the radical $\frr$. 
In particular, for all non-unimodular  centerless Lie algebras but $\aff(\R)\oplus\aff(\R)$ and $\aff(\C)$, we always have $b_{\phi}(e_7,e_7)=0$. The reason is the following. 
By \cite{AnBaDoOv}, in all of the mentioned cases $\frr$ can be written as a semidirect product $\frr=\R\oplus_\rho\fru$, 
where $\fru$ is either $\R^3$ or the three-dimensional Heisenberg algebra $\frh_3$. 
We can then consider a basis $\{e_4,e_5,e_6,e_7\}$ of $\frr$ in such a way that $\R = \langle e_4\rangle$ and $[e_5,e_7]=0=[e_6,e_7]$. 
Consequently, $d\phi=0$ gives $\phi_{i57}=0, \phi_{i67}=0,$ for $i=1,2,3$, and $\phi_{567} = d\phi(e_4,e_5,e_6,e_7) =0$. These conditions imply that $b_\phi(e_7,e_7)=0$. 

The remaining cases do not follow this pattern. In particular, $\aff(\R)\oplus\aff(\R)\cong \R\oplus_\rho\fre(1,1)$, for a suitable derivation $\rho(e_4)$ of the Lie algebra $\fre(1,1)$ 
of the group of rigid motions of the Minkowski 2-space, while $\aff(\C)\cong\R\oplus_\rho\fre(2)$, for a suitable derivation $\rho(e_4)$ of the Lie algebra $\fre(2)$ 
of the group of rigid motions of Euclidean 2-space.  
Finally, $d\phi(e_4,e_5,e_6,e_7)$ is identically zero when $\frr$ is unimodular, thus $\phi_{567}$ need not to vanish in that case. 
\end{remark}

\begin{proposition}\label{existencesl2}
The Lie algebra $\frsl(2,\R)\oplus\frr$ admits closed $\G_2$-structures when $\frr$ is unimodular and centerless.
\end{proposition}
\begin{proof}
By Theorem \ref{solvable4} and Lemma \ref{centerrad}, 
$\frr$ must be one of the following: $\frr_{4,-\frac12}$, $\frr_{4,\mu,-1-\mu}$ with $-1<\mu\leq-\frac12$, $\frr'_{4,\mu,-\frac{\mu}{2}}$ with $\mu>0$. 
To show the assertion, it is sufficient to give the expression of a closed $\G_2$-structure $\f$ for each case. 
For the sake of clarity, we also write the matrix associated with the bilinear form $b_\f$ with respect to the basis $\{e_1,\ldots,e_7\}$ of $\frsl(2,\R)\oplus\frr$. 
The inner product and the volume form induced by $\f$ are given by $g_\f = \det(b_\f)^{-\frac19}b_\f$ and $ \det(b_\f)^{\frac19}e^{1234567}$, respectively. 
\begin{enumerate}[$\bullet$]
\item {\bf radical} $ \frr=\frr_{4,-\frac12}$:
\[
\f = -e^{147}+2e^{236}+2e^{237}+e^{245}+e^{247}-2e^{125}+4e^{127}-2e^{135}-4e^{137}+e^{146}+e^{347}-e^{345}+e^{567},
\]
\[
b_\f= \left( 
\begin{array}{ccccccc} 
16&-4&-4&0&0&0&0		\\ \noalign{\medskip}
-4&12&-4&-2&0&0&0	\\ \noalign{\medskip}
-4&-4&12&2&0&0&0		\\ \noalign{\medskip}
0&-2&2&2&0&0&0		\\ \noalign{\medskip}
0&0&0&0&4&0&0		\\ \noalign{\medskip}
0&0&0&0&0&2&0		\\ \noalign{\medskip}
0&0&0&0&0&0&6
\end{array} 
\right).
\]
\vspace{10pt}\\
\item {\bf radical} $\frr=\frr_{4,\mu,-1-\mu}$, $-1<\mu\leq-\frac12$:
\[
\f = e^{236}+e^{245}+\frac12(\mu+1)e^{247}+e^{567} + \frac12(\mu+1)e^{347}-2e^{125}+e^{127}-2e^{135}-e^{137}-\mu e^{146}-e^{345},
\]
\[
b_\f =  \left( 
\begin{array}{ccccccc} 	
-4\,\mu&0&0&0&0&0&0		\\ \noalign{\medskip}
0&\mu+2&\mu&0&0&0&0		\\ \noalign{\medskip}
0&\mu&\mu+2&0&0&0&0		\\ \noalign{\medskip}
0&0&0&-{\mu}^{2}-\mu&0&0&0	\\ \noalign{\medskip}
0&0&0&0&4&0&0			\\ \noalign{\medskip}
0&0&0&0&0&-\mu&0			\\ \noalign{\medskip}
0&0&0&0&0&0&1+\mu
\end{array} \right).
\]
\vspace{10pt}\\
\item {\bf radical} $\frr= \frr'_{4,\mu,-\frac{\mu}{2}}$, $\mu>0$:
\begin{eqnarray*}
\f =	 e^{567}-e^{346}-\frac{\mu}{2} e^{347}+\frac{\mu}{2} e^{345}-\frac12 e^{246}-\frac{\mu}{4} e^{247}-\frac{\mu}{4} e^{245}+\sqrt{2}e^{236}-\sqrt{2}e^{147} +2e^{137}\\
	+\frac{\mu}{\sqrt{2}}e^{146}+e^{135}+\frac12e^{125}-e^{127},	
\end{eqnarray*}
\[
b_\f = \mu  \left( 
\begin{array}{ccccccc} 
\sqrt {2}&0&0&0&0&0&0									\\ \noalign{\medskip}
0&\frac{3\sqrt{2}}{8}&-\frac{\sqrt{2}}{4}&\frac{1}{2}&0&0&0			\\ \noalign{\medskip}
0&-\frac{\sqrt{2}}{4}&\frac{3\sqrt{2}}{2}&-1&0&0&0				\\ \noalign{\medskip}
0&\frac{1}{2}&-1&\frac{\sqrt{2}}{8}{\mu}^{2}+\frac{\sqrt{2}}{2}&0&0&0	\\ \noalign{\medskip}
0&0&0&0&\frac{1}{2}&0&0								\\ \noalign{\medskip}
0&0&0&0&0&1&0										\\ \noalign{\medskip}
0&0&0&0&0&0&1
\end{array} \right).
\]
\end{enumerate}

As every connected Lie group admitting a left-invariant Ricci-flat metric is solvable (cf.~\cite[Thm.~1.6]{Mil}), we can conclude that the closed G$_2$-structures described above are not co-closed. 
\end{proof}

\begin{remark}
From the results of \cite{Boc,Has}, we know that the simply connected solvable Lie groups with Lie algebras  
$\frr_{4,\mu,-1-\mu}$ and  $\frr'_{4,\mu,-\frac{\mu}{2}}$ admit a lattice for certain values of the parameter $\mu$.  
Moreover, the Lie group $\SL(2,\R)$ admits a lattice, too (see for instance \cite{Mos}). 
Hence, Proposition \ref{existencesl2} allows to obtain new examples of locally homogeneous compact 7-manifolds endowed with a closed $\G_2$-structure. 
\end{remark}

\section{Unimodular Lie algebras with non-trivial Levi decomposition}\label{SNoTrLevi}
In this last section, we focus on seven-dimensional unimodular Lie algebras with non-trivial Levi decomposition. 
According to the discussion in section $\S$\ref{NonSolv7}, up to isomorphism they are the Lie algebras $L_{7,2}$, $L_{7,3}^{-2}$, $L_{7,6}$, $L_{7,7}$ of Table \ref{tabLnt} 
(cf.~Remark \ref{RemUnimNonTrLevi}), and the Lie algebras of Table \ref{tabLnt2}.

\begin{proposition}\label{UNTN}
The unimodular Lie algebras $L_{7,2}$, $L_{7,6}$, and $L_{7,7}$ do not admit closed $\G_2$-structures.
\end{proposition}
\begin{proof}
Let us consider the structure equations of $L_{7,2}$, $L_{7,6}$ and $L_{7,7}$ given in Table \ref{tabLnt}. 
As in the proof of Proposition \ref{nonexistenceso3}, we first compute the expression of the generic closed 3-form $\phi$, and we fix the volume form $e^{1234567}$. 
A straightforward computation shows that for both $L_{7,2}$ and $L_{7,7}$ it holds $b_\phi(e_i,e_i)=0$ for $i=4,5,6,7$. 

To show our assertion for $L_{7,6}$, we prove that the restriction of the bilinear form $b_\phi$ to the radical $\frr=\langle e_4,e_5,e_6,e_7\rangle$ 
is neither positive nor negative definite. We assume that $b_\phi(e_i,e_i)\neq 0$, $i=1,\ldots,7$, otherwise we would get a contradiction by Lemma \ref{obstrgen}. 
Now, the matrix $B$ associated with $\left.b_\phi\right|_{\frr\times\frr}$ with respect to the basis $\{e_4,e_5,e_6,e_7\}$ has the following expression 
\[
B = 
\renewcommand\arraystretch{1.6}
\left( 
\begin {array}{cccc} 
b_{1}	&	b_{2}	&	b_{3}	&	b_{4}	\\ \noalign{\medskip}
b_{2}	&	3\,b_{3}	&	9\,b_{4}	&	\frac{4\,b_{4}b_{2}-b_{3}^{2}}{b_{1}}	\\ \noalign{\medskip}
b_{3}	&	9\,b_{4}	&	3\,{\frac{4\,b_{4}b_{2}-b_{3}^{2}}{b_{1}}}		
		&	-{\frac{2\,b_{1}b_{3}b_{4}-4\,b_{2}^{2}b_{4}+b_{3}^{2}b_{2}}{b_{1}^{2}}}\\ \noalign{\medskip}
b_{4}	&	\frac{4\,b_{4}b_{2}-b_{3}^{2}}{b_{1}}
		&	-{\frac{2\,b_{1}b_{3}b_{4}-4\,b_{2}^{2}b_{4}+b_{3}^{2}b_{2}}{b_{1}^{2}}}
		&	-{\frac {8\,b_{1}b_{4}^{2}-4\,b_{2}b_{3}b_{4}+b_{3}^{3}}{b_{1}^{2}}}
\end {array} 
\renewcommand\arraystretch{1}
\right),
\]
where the $b_i$ are homogeneous polynomials of degree 3 in the variables $\phi_{246}$, $\phi_{247}$, $\phi_{346}$, $\phi_{347}$, $\phi_{357}$.
The determinant of $B$ is non-negative:
\[
\det(B) = \frac{1}{b_{1}^{4}}\left( 27\,b_{1}^{2}b_{4}^{2}-18\,b_{1}b_{2}b_{3}b_{4}+4\,b_{1}b_{3}^{3}+4\,b_{2}^{3}b_{4}-b_{2}^{2}b_{3}^{2} \right) ^{2}.
\]
Hence, if $B$ is non-singular, its possible signatures are $(4,0)$, $(0,4)$, or $(2,2)$. We show that only the last case occurs. 

Let us begin considering $b_1>0$. 
Then, $B$ cannot be negative definite, and all its diagonal entries must be positive.  
Thus, we have the inequalities
\begin{eqnarray}
b_3	&>&	0,\label{B22}\\
4\,b_{4}b_{2}-b_{3}^{2} &>&	0, \label{B33}
\end{eqnarray}
together with $8\,b_{1}b_{4}^{2}-4\,b_{2}b_{3}b_{4}+b_{3}^{3} <0$. 
By Sylvester's criterion, the matrix $B$ is positive definite if and only if its leading principal minors are positive. This gives the following further conditions
\begin{eqnarray}
3\,b_1b_3-b_2^2	&>&	0,\label{M22}\\
27\,b_{1}^{2}b_{4}^{2}-18\,b_{1}b_{2}b_{3}b_{4}+4\,b_{1}b_{3}^{3}+4\,b_{2}^{3}b_{4}-b_{2}^{2}b_{3}^{2}&<&0.\label{M33}
\end{eqnarray}
Assume that $b_1>0$ and $b_3>0$ are given. Equation \eqref{B33} implies $b_2\,b_4>0$, while equation \eqref{M22} gives
$-\sqrt{3\,b_1b_3} < b_2 <\sqrt{3\,b_1b_3}$. 
Under these constraints on $b_1$, $b_2$ and $b_3$, \eqref{M33} becomes a quadratic inequality in the variable $b_4$. 
Since its discriminant is $-16(3\,b_1b_3-b_2^2)^3$, it does not admit real solutions. Hence, the signature of $B$ is $(2,2)$. 

When $b_1<0$, a similar argument leads to the same result. 
Moreover, if we fix the volume form $-e^{1234567}$ instead of $e^{1234567}$, then the signature of $B$ does not change. This completes the proof. 
\end{proof}

\begin{proposition}\label{UNTY}
The unimodular Lie algebra $L_{7,3}^{-2}$ admits closed $\G_2$-structures.
\end{proposition}
\begin{proof}
A closed $\G_2$-structure is given for instance by the stable 3-form
\[
\f = e^{157}-e^{235}-e^{267}-3e^{124}-e^{126}+e^{134}-e^{136}-e^{456}+e^{367}+e^{247}.
\]
The matrix associated with the corresponding bilinear form $b_\f$ with respect to the basis $\{e_1,\ldots,e_7\}$ of $L_{7,3}^{-2}$ and the volume form $e^{1234567}$ is 
\[
\left( 
\begin {array}{ccccccc} 
4&0&0&0&0&0&-\frac32	\\ \noalign{\medskip}
0&4&-\frac32&0&0&0&0	\\ \noalign{\medskip}
0&-\frac32&1&0&0&0&0	\\ \noalign{\medskip}
0&0&0&1&0&\frac12&0	\\ \noalign{\medskip}
0&0&0&0&1&0&0		\\ \noalign{\medskip}
0&0&0&\frac12&0&2&0	\\ \noalign{\medskip}
-\frac32&0&0&0&0&0&1
\end {array} \right).
\] 
As $L_{7,3}^{-2}$ is not solvable, $\f$ is closed but not co-closed (cf.~\cite[Thm.~1.6]{Mil}). 
\end{proof}

Notice that all of the Lie algebras appearing in Table \ref{tabLnt2} have non-trivial center. Hence, they cannot admit left-invariant closed $\G_2$-structures by 
Proposition \ref{ClosedSympl} and the following result. 
\begin{theorem}[\cite{Chu,LicMed}]
A unimodular Lie algebra admitting a symplectic structure is solvable. 
\end{theorem}

\noindent  {\bf Acknowledgements.} 
The authors would like to thank Fabio Podest\`a for useful conversations, and the two anonymous referees for their valuable comments. 
This work was done when A.~R.~was a postdoctoral fellow at the Department of Mathematics and Computer Science ``U.~Dini'' of the Universit\`a degli Studi di Firenze. 

\appendix

\section{}\label{appendix}

In this appendix, we give some details on the computations we did to prove propositions \ref{nonexistenceso3}, \ref{nonexistencesl2}, \ref{existencesl2}, \ref{UNTN}, \ref{UNTY}. 
We focus on the case $\frg = \mathfrak{so}(3)\oplus\aff(\R)\oplus\aff(\R)$, as in the remaining cases one can proceed similarly. 

Let $\{e_1,\ldots,e_7\}$ be the basis of $\frg=\mathfrak{so}(3)\oplus\aff(\R)\oplus\aff(\R)$ described in the proof of Proposition \ref{nonexistenceso3}. 
Then, the structure equations of $\frg$ with respect to the dual basis $\{e^1,\ldots,e^7\}$ are the following:
\[
\left(-e^{23},e^{13},-e^{12},0,-e^{45},0,-e^{67}\right).
\]

Let us consider a generic 3-form $\phi = \sum_{1\leq i<j<k\leq7} \phi_{ijk}e^{ijk}\in\Lambda^3(\frg^*)$, where $\phi_{ijk}\in\R$. 
The condition $d\phi=0$ is equivalent to the following system of linear equations in the variables $\{\phi_{ijk}\}$: 
\[
\left\{
\begin{array}{lc}
\phi_{i46}=0,~\phi_{i47}=0,~\phi_{i56}=0,~\phi_{i57}=0, &\quad i=1,2,3,\\ 
\phi_{125}+\phi_{345}=0,~\phi_{127}+\phi_{367}=0,& \\
\phi_{135}-\phi_{245}=0,~\phi_{137}-\phi_{267}=0, &\\
\phi_{145}+\phi_{235}=0,~\phi_{237}+\phi_{167}=0,& \\
\phi_{567}+\phi_{457}=0.& 
\end{array}
\right.
\]
Solving this system, we obtain the following expression for the generic closed 3-form $\phi$ on $\frg$
\begin{eqnarray*}
\phi	&=& \phi_{123}e^{123} +\phi_{124}e^{124} -\phi_{345}e^{125} +\phi_{126}e^{126} -\phi_{367}e^{127}  +\phi_{134}e^{134} +\phi_{245}e^{135}  \\
	&  &+\phi_{136}e^{136} +\phi_{267}e^{137}  -\phi_{235}e^{145} -\phi_{237}e^{167}+\phi_{234}e^{234} +\phi_{235}e^{235} +\phi_{236}e^{236} \\ 
	&  & +\phi_{237}e^{237}+\phi_{245}e^{245} +\phi_{267} e^{267} +\phi_{345}e^{345} + \phi_{367}e^{367}+\phi_{456}e^{456}-\phi_{567}e^{457}\\
	&  & +\phi_{467}e^{467}+\phi_{567}e^{567}.
\end{eqnarray*}
Using this, we compute $\iota_{e_i}\phi$, $i=1,\ldots,7,$ and we observe that 
\begin{eqnarray*}
b_\phi(e_5,e_5)		&=&	- \phi_{567}\left( {\phi_{235}}^2+{\phi_{245}}^2+{\phi_{345}}^2 \right) e^{1234567}, \\
b_\phi(e_7,e_7)		&=& \phi_{567} \left({\phi_{237}}^2+{\phi_{267}}^2+{\phi_{367}}^2\right) e^{1234567}. 
\end{eqnarray*}
This implies that $b_\phi$ cannot be definite.


\end{document}